\documentclass[11pt]{amsart}

\usepackage{amsmath,amsthm,amsfonts,amssymb,bm,graphicx }

\usepackage
{hyperref}
\hypersetup{colorlinks=true,citecolor=blue,linkcolor=blue,urlcolor=blue}

\theoremstyle{plain}
\newtheorem{theorem}{Theorem}
\newtheorem{lemma}[theorem]{Lemma}

\newtheorem{cor}[theorem]{Corollary}
\newtheorem{prop}[theorem]{Proposition}
\numberwithin{equation}{section}

\theoremstyle{definition}

\newcommand{\pr}{^\prime}

\newcommand{\Z}{\mathbb{Z}}
\newcommand{\Q}{\mathbb{Q}}

\newcommand{\co}{\mathcal{O}}

\makeatletter
\DeclareRobustCommand\widecheck[1]{{\mathpalette\@widecheck{#1}}}
\def\@widecheck#1#2{%
    \setbox\z@\hbox{\m@th$#1#2$}%
    \setbox\tw@\hbox{\m@th$#1%
       \widehat{%
          \vrule\@width\z@\@height\ht\z@
          \vrule\@height\z@\@width\wd\z@}$}%
    \dp\tw@-\ht\z@
    \@tempdima\ht\z@ \advance\@tempdima2\ht\tw@ \divide\@tempdima\thr@@
    \setbox\tw@\hbox{%
       \raise\@tempdima\hbox{\scalebox{1}[-1]{\lower\@tempdima\box
\tw@}}}%
    {\ooalign{\box\tw@ \cr \box\z@}}}
\makeatother

\newcommand{\ve}{\epsilon}

\begin{document}

\author{Valentin Blomer}
\address{University of G\"ottingen, Mathematisches Institut, Bunsenstr.~3-5, D-37073 G\"ottingen, Germany} \email{vblomer@math.uni-goettingen.de}

\author{V\' \i t\v ezslav Kala}
\address{University of G\"ottingen, Mathematisches Institut, Bunsenstr.~3-5, D-37073 G\"ottingen, Germany}
\address{Charles University, Faculty of Mathematics and Physics, Department of Algebra, Sokolov\-sk\' a 83, 18600 Praha~8, Czech Republic}
\email{vita.kala@gmail.com}

\title{On the rank of universal quadratic forms over real quadratic fields}
 
\thanks{Second author was partially supported by Czech Science Foundation GA\v CR, grant 17-04703Y}


\begin{abstract}
We study the minimal number of variables required by a totally positive definite diagonal universal quadratic form over a real quadratic field $\Q(\sqrt D)$ and obtain lower and upper bounds for it in terms of certain sums of coefficients of the associated continued fraction. 
We also estimate such sums  in terms of $D$ and establish a link between continued fraction expansions and special values of $L$-functions in the spirit of Kronecker's limit formula. 
\end{abstract}


\setcounter{tocdepth}{2}  

\maketitle 

\section{Introduction}

The story of universal quadratic forms began with the four square theorem that was proved by Lagrange and states that the quadratic form $x^2+y^2+z^2+w^2$
represents all positive integers. This was followed by a large number of further results over the ring of integers $\Z$, such as the classification of 
all universal quaternary diagonal forms by Ramanujan and Dickson,   the 15-theorem of Conway, Miller and Schneeberger (see \cite{Bh} for a beautiful proof) and the 290-theorem of  Bhargava and Hanke \cite{BH}.

Maa{\ss} \cite{Ma} proved that the sum of three squares is universal over $\Q(\sqrt 5)$, and then Siegel \cite{Si} showed that over every other number field, 
the sum of any number of squares is not universal. This turned the attention to other quadratic forms: 
Chan, Kim, and Raghavan \cite{CKR} studied ternary universal forms over real quadratic fields, and for example showed that diagonal ternary universal forms exist only over $\Q(\sqrt 2)$, $\Q(\sqrt 3)$, and $\Q(\sqrt 5)$. 
There were numerous other results over specific real quadratic fields, e.g., Sasaki \cite{Sa} found all universal quaternary forms over $\Q(\sqrt {13})$.
For more general discriminants, Kim \cite{Ki, Ki2} proved that there are only finitely many real quadratic fields that have a diagonal universal form in 7 variables, but also constructed an 8-ary diagonal universal form over each field $\Q(\sqrt{n^2-1})$ (when $n^2-1$ is squarefree).

For a general number of variables $m$, the present authors \cite{BK, Ka} constructed infinitely many real quadratic fields that do not admit $m$-ary universal quadratic forms. Using very different techniques, Yatsyna \cite{Ya} recently extended these results to the case of
number fields that possess units of every signature.

However, still not much is known about the number of variables required by a general (diagonal) universal form over a real quadratic field: this is the goal of the present paper. We will focus on totally positive definite diagonal quadratic forms of arity (or rank) $m$ over real quadratic fields $K = \Q(\sqrt D)$, i.e., 
forms $Q(x_1, \dots, x_m)=a_1x_1^2+\dots+a_mx_m^2$, where $a_1, \dots, a_m$ are totally positive elements of the ring of integers $\co_K$.
Such a form is universal if it represents all totally positive integers $\alpha$, i.e., if $\alpha=Q(x_1, \dots, x_m)$ for some $x_1, \dots, x_m\in\co_K$. Let us denote by $m_{\mathrm{diag}}(D)$ the smallest integer $m$ such that there is an $m$-ary diagonal universal form over $K$. We introduce the following notation. For squarefree $D > 1$ let 
\begin{equation}\label{omegaD}
\omega_D=\begin{cases}
\sqrt D &\mbox{if } D\equiv 2,3\\
\frac{1+\sqrt D}2&\mbox{if } D\equiv 1
\end{cases}\pmod 4,
\end{equation}
$\Delta \in \{D, 4D\}$ the discriminant of $\co_K$ and $\omega_D=[u_0, \overline{u_1, u_2, \dots, u_{s-1}, u_s}]$   the periodic continued fraction expression  with $s$ minimal. We define
\begin{equation}\label{MD}
M_D=\begin{cases}
u_1+u_3+\dots+u_{s-1}&\mbox{if } s\mbox{ is even,}\\
2u_0+u_1+u_2+\dots+u_{s-1}&\mbox{if } s\mbox{ is odd and } D\equiv 2,3\pmod 4,\\
2u_0+u_1+u_2+\dots+u_{s-1}-1&\mbox{if } s\mbox{ is odd and } D\equiv 1\pmod 4.
\end{cases}
\end{equation}
If $s$ is odd, equivalently if $\co_K$ has a unit with negative norm, then 
\begin{equation}\label{s odd}
M_D =u_1 + \ldots + u_s. 
\end{equation}
For $\varepsilon>0$ define $M^*_{D, \varepsilon}$ as the sum in \eqref{MD}, but ranging only over coefficients $u_i\geq D^{1/8+\varepsilon}$.

Our first result is then the following:

\begin{theorem}\label{main 1}
With the above notation we have 
$$\max\left(\frac{M_D}{\kappa s}, C_\varepsilon M^*_{D, \varepsilon}\right)\leq m_{\mathrm{diag}}(D)\leq 8M_D$$
for any $\varepsilon > 0$, where $C_\varepsilon>0$ is a constant (depending only on $\varepsilon$) and $\kappa = 2$ if $s$ is odd and $\kappa = 1$ otherwise.
\end{theorem}

We will prove this result as Theorems \ref{construct universal}, \ref{lower bound}, and \ref{lower C}. 
All of these proofs are based on studying the representability of additively indecomposable integers of $\Q(\sqrt D)$, which can be nicely characterized in terms of convergents and semi-convergents to the corresponding continued fraction.
To show the upper bound, we generalize Kim's result \cite{Ki2} and construct an explicit diagonal universal form, whose coefficients are (some of) these indecomposable integers. The lower bound hinges on showing that sufficiently many indecomposables (essentially) have to appear as the coefficients of any diagonal universal form. Comparing our two lower bounds $\frac {M_D}{\kappa s}$ and $ C_\varepsilon M^*_{D, \varepsilon}$, the first one is clearly larger for small values of $D$ or $s$.
However, one can estimate the size of $s$ (cf.\   \eqref{estimate s}) to see that the second bound is larger for most $D$'s.

Of course one would like to know how $M_D$ and $M_{D, \varepsilon}^{\ast}$ behave asymptotically with respect to $D$ or equivalently $\Delta$. In this respect, we will prove in Section \ref{sec:main estimate} the following. 
\begin{theorem}\label{main 1a}
We have $M_D \leq c \sqrt{\Delta} (\log \Delta)^2$ for an absolute constant $c>0$. If $s$ is odd, equivalently if $ \co_K$ has a unit of negative norm, we have
$M_{D, \varepsilon}^{\ast} \geq \sqrt{\Delta}$
for every $\varepsilon< 1/8$. 
 \end{theorem}

Under the Generalized Riemann Hypothesis, the upper bound can be sharpened somewhat (cf.\ \eqref{GRH}). The lower bound for odd $s$  follows simply from the fact that $u_0 = \lfloor\omega_D\rfloor \geq \frac 12 \sqrt{\Delta}$. 

With more analytic work, one can even establish an asymptotic formula for the sum of convergents, which may be of independent interest. It  can be viewed as a variation of Kronecker's limit formula for real quadratic fields and  highlights  the fascinating connection between special $L$-values and continued fractions. The following result is proved in Corollary \ref{Korollar}, which in view of \eqref{s odd} gives an asymptotic formula for $M_D$ if $s$ is odd. 

\begin{theorem}\label{main 2}
As $D \rightarrow \infty$, we have 
$$\sum_{i=1}^s u_i \sim  \frac{ \sqrt{\Delta} }{\zeta^{(\Delta)}(2)}   \Bigl(L(D)  +  \frac{1}{h} L(1, \chi_{\Delta})\log \sqrt{D} \Bigr), $$
where $h=h_D$ is the class number, $\zeta^{(\Delta)}(s)$ is the Riemann zeta function with Euler factors at primes dividing $\Delta$ removed, $\chi_{\Delta}$ is the usual quadratic character associated with the fundamental discriminant $\Delta$, and $L(D)$ (defined in \eqref{LD}) is the constant Taylor coefficient of the $\zeta$-function associated to the class of principal ideals. 
\end{theorem}

Kronecker's limit formula is concerned with finding a closed expression for $L(D)$ (and more general functions). Zagier and Hirzebruch observed that for real quadratic fields there is a connection between $L(D)$ and the coefficients of the continued fraction of $\omega_D$. The exact formula in \cite[Corollary 2]{Za} (which is derived by a completely different method than Theorem \ref{main 2}), however, seems to be hard to use to obtain any sort of asymptotic statement.  The beautiful formula \cite[Satz 2, \S 14]{Za2}, on the other hand, is of different nature, since it treats the alternating sum $\sum_{i=1}^s (-1)^i u_i$, cf.\ \cite[p.\ 131]{Za2} (and gives in particular no information if $s$ is odd).  Yet another variation of the connection between special values of class group $L$-functions for real quadratic fields and continued fractions can be found in a nice paper of Bir\'o and Granville  \cite[Theorem 1]{BG}. \\

Most of what we do in this paper can probably be generalized to non-maximal orders in $\mathcal{O}_K$, but in order to avoid technical subtleties in particular in the analytic argument, we concentrate on the most natural case of the maximal order. \\

\textbf{Acknowledgement:} We would like to thank the referee for useful suggestions that improved and simplified the presentation.

\section{Preparation: Indecomposables}\label{sec:prep}

\subsection{Basic notation} Throughout the paper, we shall use the following notation:
Let $K=\mathbb Q(\sqrt D)$ with squarefree $D>1$ and discriminant $\Delta$. For $\alpha=x+y\sqrt D\in\Q(\sqrt D)$ we denote its conjugate as $\alpha\pr=x-y\sqrt D\in\Q(\sqrt D)$ 
and its norm as $N(\alpha)=x^2-Dy^2$.
An element $\alpha$ is totally positive if $\alpha>0$ and $\alpha\pr>0$; if $\alpha-\beta$ is totally positive, we write $\alpha\succ\beta$. We denote by $\co_K$ the ring of integers of $K$ and by $\co_K^+$ the semiring of totally positive integers. Let $\omega_D$ be as in \eqref{omegaD}, so that
 $$-\omega_D\pr=\begin{cases}
\sqrt D &\mbox{if } D\equiv 2,3\\
\frac{-1+\sqrt D}2&\mbox{if } D\equiv 1
\end{cases}\pmod 4,$$
and $\mathcal O_K=\mathbb Z[\omega_D]$. We recall from the introduction that $\omega_D=[u_0, \overline{u_1, u_2, \dots, u_{s-1}, u_s}]$. Note that $u_s=2u_0$ when $D\equiv 2,3\pmod 4$ and $u_s=2u_0-1$ when $D\equiv 1\pmod 4$ (which we used to go from \eqref{MD} to \eqref{s odd}). We also know that the sequence $(u_1, u_2, \dots, u_{s-1})$ is symmetric, i.e., $u_i=u_{s-i}$.\\
 
Let $\frac {p_i}{q_i}:=[u_0, \dots, u_i]$ be the $i$th convergent to $\omega_D$ (where $p_i, q_i$ are coprime positive integers), 
and $\alpha_i=p_i-q_i\omega_D\pr$ the corresponding element of $\mathcal O_K$, which by small abuse of notation we also call a convergent.  We have 
$p_{i+1}=u_{i+1}p_i+p_{i-1}$, $q_{i+1}=u_{i+1}q_i+q_{i-1}$, and 
$\alpha_{i+1}=\alpha_{i}q_i+\alpha_{i-1}$
(with initial conditions $p_{-1}:=1$, $p_0=k$, $q_{-1}:=0$, $q_0=1$). Note that $\alpha_{-1}=1$,
and that $\alpha_i\succ 0$ if and only if $i$ is odd.  \\

Let $\ve$ be the \emph{totally positive} fundamental unit, $N(\ve)=1$. 
Then we have $\ve=\alpha_{s-1}$ or $\ve=\alpha_{2s-1}$ when $s$ is even or odd, respectively.\\

By a semi-convergent to $\omega_D$ we mean a fraction of the form $\frac{p_i+rp_{i+1}}{q_i+rq_{i+1}}$ with $0\leq r\leq u_{i+2}$ ($i\geq -1$). 
Note that each convergent is also a semi-convergent and that if we take $r=u_{i+2}$, the fraction is just $p_{i+2}/q_{i+2}$. 
We denote the corresponding element of $\mathcal O_K$ (which we also call a semi-convergent) by $\alpha_{i, r}:=\alpha_i+r\alpha_{i+1}$ (again for $i\geq -1$ and $0\leq r\leq u_{i+2}$). A semi-convergent $\alpha_{i, r}$ is a convex combination of $\alpha_i$ and $\alpha_{i+2}$. 
These are either both totally positive, or none of them is, and 
hence $\alpha_{i, r}$ is totally positive if and only if $i$ is odd.\\

Denote by $S$ the set of all totally positive semi-convergents and their conjugates, i.e., 
$$S:=\{\alpha_{i, r}, \alpha_{i, r}\pr \mid i\geq -1 \mathrm{\ odd}, 0\leq r< u_{i+2}\},$$ 
and by $S_0$ its subset of elements $\sigma\in S$ satisfying $\ve>\sigma\geq\sigma\pr>0$ (note that in the definition of $S$, all the elements are distinct, except for $\alpha_{-1, 0}=1=\alpha_{-1, 0}\pr$).
The set $S$ is closed under multiplication by $\ve^k$ for any $k\in\mathbb Z$, and 
each element of $S$ can be uniquely written in the form $\sigma_0\ve^k$ for some $\sigma_0\in S_0$ and $k\in\mathbb Z$.

The cardinality of $S_0$ is clearly $u_1+u_3+\dots+u_{s-1}$ if $s$ is even and $u_1+u_3+\dots+u_s+u_{s+2}+\dots+u_{2s-1}$ if $s$ is odd,
because $S_0$ consists precisely of all the elements $\alpha_{i, r}$ with odd $-1\leq i\leq s-3$ ($2s-3$, resp.). Note that this equals $M_D$ as defined in \eqref{MD}, as $u_s=2u_0$ or $u_s=2u_0-1$ when $D\equiv 2, 3\pmod 4$ or $1\pmod 4$ and $u_{s+i}=u_i$ for $i\geq 1$.\\

A totally positive integer $\alpha\in\mathcal O_K^+$ is (additively) indecomposable if it cannot be decomposed as 
the sum of two totally positive elements, i.e., if $\alpha\neq\beta+\gamma$ for $\beta, \gamma\in\mathcal O_K^+$. 
By a classical theorem 
(see, e.g., \cite[\S 16]{Pe} or \cite[Theorem 2]{DS}), semi-convergents (and their conjugates) are exactly all indecomposable elements, i.e., $S$ is the set of all indecomposable integers.\\

In later sections we will use the following asymptotic notation: For real functions $f(x), g(x)$, we write $f\sim g$ if $\lim_{x\rightarrow\infty} f(x)/g(x)=1$, $f\ll g$ (or $g\gg f$) if there are $c>0$ and $x_0$ such that $|f(x)|<cg(x)$ for all $x>x_0$, 
and $f\asymp g$ if $f\ll g$ and $f\gg g$.

\subsection{Estimates} We will  need certain estimates on the sizes of indecomposables and their norms. These are mostly classical; in the case $D\equiv 2, 3\pmod 4$ they have appeared for example in \cite{Ka2}. We will need them also 
when $D\equiv 1\pmod 4$, and so we collect the required results here, giving proofs only for this case. 

We will first need to introduce some additional notation. For a convergent $\alpha_i$, we set
$$N_i:=|N(\alpha_i)|=(-1)^{i+1}N(\alpha_i)= \begin{cases} |p_i^2-p_iq_i-q_i^2\frac{D-1}4|, & D \equiv 1 \, (\text{mod }4),\\   |p_i^2-Dq_i^2|, & D \equiv 2, 3 \, (\text{mod }4). \end{cases} $$ 
Recall that we have $p_{i+1}q_i-p_iq_{i+1}=(-1)^i$ and let us define $T_i$ so that $\alpha_{i-1}\alpha_{i}\pr=T_{i}+(-1)^{i+1}\omega_D$, i.e., we have $T_i=p_i(p_{i-1}-q_{i-1})-q_iq_{i-1}\frac{D-1}4$ or
$T_i=p_ip_{i-1}-Dq_iq_{i-1}$ when $D\equiv 1\pmod 4$ or $2, 3\pmod 4$, respectively.  
Finally, set
$c_i:=[u_i, u_{i+1}, u_{i+2}, \dots]$ so that 
\begin{equation}\label{formula}
c_i=u_i+\frac 1{c_{i+1}}, \quad \omega_D=\frac{c_{i+1}p_i+p_{i-1}}{c_{i+1}q_i+q_{i-1}}.
\end{equation}

Concerning norms of indecomposables, first note that 
as in the proof of \cite[Proposition 1]{Ka2} (or \cite[Theorem 4]{JK}), for odd $i$ we have 
\begin{equation}\label{norms}
N(\alpha_{i,r})=\frac{(D-1)/4+(T_{i+1}-N_{i+1}r)-(T_{i+1}-N_{i+1}r)^2}{N_{i+1}} \quad \text{or} \quad  
\frac{D-(T_{i+1}-N_{i+1}r)^2}{N_{i+1}}
\end{equation}
 (when $D\equiv 1$ or $2,3\pmod 4$). In order to better estimate the sizes of norms, we will use the following two lemmas.

\begin{lemma}\label{prop 5}
For $i\geq 0$ we have
\begin{equation}\label{lemma3}
T_i=(-1)^{i}\left(\omega_D-\frac{N_{i-1}}{c_{i+1}}\right) \quad \text{and} \quad 
N_i=\frac{ \sqrt \Delta}{c_{i+1}}-\frac{N_{i-1}}{c_{i+1}^2}.
\end{equation}
\end{lemma}

\begin{proof}
As we indicated above, we will give the proof only in the case $D\equiv 1\pmod 4$; for the other case see \cite[Proposition 5]{Ka2}. From the second formula in \eqref{formula} we conclude $c_{i+1}\alpha_i\pr=-\alpha_{i-1}\pr$. Multiplying by $-\alpha_i$ we get 
\begin{equation}\label{first}
-N(\alpha_i)c_{i+1}=\alpha_i\alpha_{i-1}\pr=T_{i}+(-1)^{i+1}\omega_D\pr.
\end{equation} 
Moreover, we have 
\begin{displaymath}
\begin{split}
T_{i+1} & =p_{i+1}(p_{i}-q_{i})-q_iq_{i+1}\frac{D-1}4=
(u_{i+1}p_i+p_{i-1})(p_{i}-q_{i})-(u_{i+1}q_i+q_{i-1})q_i\frac{D-1}4= \\
& =u_{i+1}N(\alpha_i)+T_i+p_iq_{i-1}-p_{i-1}q_i=u_{i+1}N(\alpha_i)+T_i+(-1)^{i+1}.
\end{split}
\end{displaymath}
Plugging in \eqref{first}, we obtain
$$
(-1)^{i}\omega_D\pr-N(\alpha_{i+1})c_{i+2} =u_{i+1}N(\alpha_i)+(-1)^{i+1}\omega_D\pr-N(\alpha_i)c_{i+1}+(-1)^{i+1},
$$
and so 
$$
N(\alpha_{i+1})c_{i+2} =(-1)^{i}(2\omega_D\pr+1)+N(\alpha_i)(c_{i+1}-u_{i+1})=(-1)^{i}\sqrt D+\frac{N(\alpha_i)}{c_{i+2}}.
$$
This proves the formula for $N_i$ in \eqref{lemma3}. The formula for $T_i$ follows by plugging in the expression for $N_i$ into \eqref{first}.
\end{proof}

\begin{lemma}\label{prop 6}
For $i\geq 0$ we have 
\begin{equation}\label{4a}
\frac{ \sqrt \Delta}{c_{i+1}}\left(1-\frac 1{c_ic_{i+1}}\right)<N_i<\frac{ \sqrt \Delta}{c_{i+1}}
\end{equation}
and in particular
\begin{equation}\label{4b}
 \frac{1}{u_{i+1}+10}  <\frac{N_i}{ \sqrt{\Delta}} \,\,\, \text{ if } u_{i+1} \geq 3\quad \text{and} \quad  \frac{N_i}{ \sqrt{\Delta}} < \frac{1}{u_{i+1}}.
\end{equation}
 \end{lemma}
 
\begin{proof} The upper bound in \eqref{4a} follows immediately from the second formula in \eqref{lemma3}. The lower bound in the case $D \equiv 1$ (mod 4)  follows from
$$N_i=\frac{\sqrt \Delta}{c_{i+1}}-\frac{N_{i-1}}{c_{i+1}^2}=
\frac{\sqrt \Delta}{c_{i+1}}-\frac{1}{c_{i+1}^2}\left(\frac{\sqrt \Delta}{c_{i}}-\frac{N_{i-2}}{c_{i}^2}\right)>\frac{\sqrt \Delta}{c_{i+1}}\left(1-\frac 1{c_ic_{i+1}}\right).$$
The   bounds in \eqref{4b} are simple consequences of \eqref{4a} and  the first formula in \eqref{formula}, cf.\ e.g.\ \cite[Theorem 8a]{Ka2}. 
\end{proof}

Finally, from \cite[Theorem 5.9]{JW} we conclude the following useful result. Note that by the discussion \cite[Section 5.2/5.3]{JW}, principal reduced ideals correspond to convergents. 

\begin{prop}\label{norm implies alpha} 
Assume that $\mu\in\mathcal O_K$ is such that $0<|N(\mu)|<\frac{1}{2}  \sqrt \Delta$. Then $\mu=n\alpha_i$ or $\mu=n\alpha_i\pr$ for some $i\geq -1$ and $n\in\mathbb Z$.
\end{prop}

Note that  
Lemma \ref{prop 6} and Proposition \ref{norm implies alpha} together give quite a close correspondence between primitive ideals of small norm and coefficients of the continued fraction:
The elements $\alpha_i$, $i = 0, 1, \ldots, s-1$ generate primitive (i.e., not divisible by a rational integer other than 1) principal ideals of norm $<  \sqrt{\Delta}$. 
Conversely, the generators $\alpha$ of all primitive principal ideals  $\alpha$   with  $|N\alpha| <  \sqrt{\Delta}/2$ come from a convergent to the continued fraction expansion of $\omega_{D}$, 
so that the primitive principal ideals of norm $<  \sqrt{\Delta}/2$ can naturally be injected into the set of $\{u_i \mid 1 \leq i \leq s\}$; 
the element $\alpha_i$ corresponds to the coefficient $u_{i+1}$.

Since the convergent $\alpha_j$ is totally positive if and only if $j$ is odd, we also obtain an analogous correspondence between elements with negative norm modulo totally positive units and coefficients $u_i$ with odd $i$, where $1\leq i\leq s$ if $s$ is even and  $1\leq i\leq 2s$ if $s$ is odd. In terms of estimating the sum of coefficients $u_i$, we summarize this discussion as follows: 

\begin{cor}\label{estimate sum}
{\rm a)} We have 
$$\underset{N\mathfrak{a}<  \sqrt{\Delta}/2}{\left.\sum \right.^{\ast}}  \frac{ \sqrt{\Delta}}{N\mathfrak{a}}  + O\Bigl(\underset{N\mathfrak{a} <  \sqrt{\Delta}}{\left.\sum \right.^{\ast}} 1 \Bigr)<
\sum_{i=1}^s u_i< 
\underset{N\mathfrak{a} <   \sqrt{\Delta}}{\left.\sum \right.^{\ast}}
\frac{ \sqrt \Delta}{N\mathfrak a},$$
where $\sum^*$ denotes the sum over principal primitive ideals.

{\rm b)} We have 
$$\underset{N\mathfrak{a}<  \sqrt{\Delta}/2}{\left.\sum \right.^{-}}  \frac{\sqrt{\Delta}}{N\mathfrak{a}}  + O\Bigl(\underset{N\mathfrak{a} <   \sqrt{\Delta}}{\left.\sum \right.^{-}} 1 \Bigr)<
M_D<
\underset{N\mathfrak{a} <
  \sqrt{\Delta}}{\left.\sum \right.^{-}} 
\frac{ \sqrt \Delta}{N\mathfrak a},$$
where $\sum^-$ denotes the sum over principal primitive ideals generated by an element with negative norm.
\end{cor}

\begin{proof} The upper bound follows easily from the previous remark and the second bound in \eqref{4b}. For the lower bound we use the first bound in \eqref{4b} to conclude that
$$\sum_{i=1}^s u_i  >  \sum_{\substack{i \leq s\\ u_{i} \geq 3}}\left(\frac{\sqrt{\Delta}}{N_{i-1}} - 10\right) \geq  \underset{N\mathfrak{a}<  \sqrt{\Delta}/2}{\left.\sum \right.^{\ast}}  \frac{ \sqrt{\Delta}}{N\mathfrak{a}}  - 10 \underset{N\mathfrak{a} <  \sqrt{\Delta}}{\left.\sum \right.^{\ast}} 1 -  \sum_{\substack{i \leq s\\ u_{i} \leq 2}} \frac{\sqrt{\Delta}}{N_{i-1}} .$$
Using again the first  bound in \eqref{4b}, we have $\sqrt{\Delta}/N_{i-1} \leq u_{i} + 10$, so that we obtain
$$\sum_{i=1}^s u_i > \underset{N\mathfrak{a}<  \sqrt{\Delta}/2}{\left.\sum \right.^{\ast}}  \frac{ \sqrt{\Delta}}{N\mathfrak{a}}  - 22 \underset{N\mathfrak{a} <  \sqrt{\Delta}}{\left.\sum \right.^{\ast}} 1$$
as desired. The lower bound in b) is proved similarly. 
\end{proof}

\section{Upper bound}

In this section we will prove an upper bound on the minimal number of variables of universal quadratic forms over $\Q(\sqrt D)$ by constructing an
explicit diagonal universal form that generalizes the results of Kim \cite{Ki2}. 
The idea is to first express a totally positive integer as the sum of indecomposables. Since the set $S_0$ of indecomposables modulo (squares of) units is finite, we can collect together the terms corresponding to the same element of $S_0$, obtaining certain sums of units as coefficients. 

Denote by $E = \mathbb{N}_0[\ve, \ve^{-1}]$ the semiring of all elements $\sum_{i=i_0}^{i_1} e_i\ve^i$ with $i_0, i_1\in\mathbb Z$ and $e_i\in\mathbb Z$, $e_i\geq 0$
(recall that $\ve$ denotes the totally positive fundamental unit). Note that since $\ve$ is totally positive, we have that $E\subset \mathcal O_K^+\cup\{0\}$.

\begin{prop}\label{express using semi-convergents}
a) Every element $a\in\mathcal O_K^+$ is of the form $a=\sum_j n_j\sigma_j$ for some $n_j\geq 0$ and $\sigma_j\in S$.

b) Every element $a\in\mathcal O_K^+$ is of the form $a=\sum_{\sigma\in S_0} e_\sigma\sigma$ for some $e_\sigma\in E$.
\end{prop}

\begin{proof}
a) This is clear, as $S$ is the set of all indecomposable elements (and there are no infinite chains $a_1\succ a_2\succ a_3\succ\dots$ of totally positive integers).

b) follows from a) by noting that if $\sigma\in S$, then there is some $k\in\mathbb Z$ such that $\sigma\ve^k\in S_0$.
\end{proof}

We can now simplify the sums of units as follows.

\begin{lemma}\label{kim}
a) For $e\in E$ there are $i, c, d\in\mathbb Z$ with $c,d\geq 0$ such that $e=c\ve^i+d\ve^{i+1}$.  

b) All elements $e\in E$ are represented by the 8-ary form $x_1^2+x_2^2+x_3^2+x_4^2+\ve(x_5^2+x_6^2+x_7^2+x_8^2)$.
\end{lemma}

\begin{proof}
The proof is similar to the proofs of Lemma 4 and Theorem 1 of \cite{Ki2}, but let us include it for completeness.

Since $N(\ve)=1$, we have $\ve^2=A\ve-1$ with $A=\mathrm{Tr\ }\ve> 2$. By induction, this establishes the following claim: 
for each $\ell\geq 2$ we have
\begin{equation}\label{claim}
\ve^\ell=-1+\sum_{j=1}^{\ell-1} b_j\ve^j \quad \text{with} \quad  b_1>0, \quad   b_j\geq 0 \text{ for } j\geq 2.
\end{equation}
Indeed, if \eqref{claim} holds,  then
$$\ve^{\ell+1}=-\ve+\ve^2+(b_1-1)\ve^2+\sum_{j=2}^{\ell-1} b_j\ve^{j+1}=-1+(A-1)\ve+(b_1-1)\ve^2+\sum_{j=2}^{\ell-1} b_j\ve^{j+1}$$
 as desired.  Let us now rewrite the expression $e=\sum_{i=i_0}^{i_1} e_i\ve^i$ using this claim for $\ell:=i_1-i_0$.
In the case when $e_{i_0}\geq e_{i_1}$, we have
\begin{equation*}
\begin{split}
e& =\sum_{i=i_0}^{i_1-1} e_i\ve^i+e_{i_1}\ve^{i_0+\ell}=
\sum_{i=i_0}^{i_1-1} e_i\ve^i+e_{i_1}\ve^{i_0}\cdot\Bigl(-1+\sum_{j=1}^{\ell-1} b_j\ve^j\Bigr)= \\
& =
(e_{i_0}- e_{i_1})\ve^{i_0}+\sum_{i=i_0+1}^{i_1-1} (e_i+e_{i_1}b_{i-i_0})\ve^i.
\end{split}
\end{equation*}
When $e_{i_0}\leq e_{i_1}$, we similarly have 
$$e=\sum_{i=i_0+1}^{i_1} e_i\ve^i+e_{i_0}\ve^{i_0}\cdot\Bigl(-\ve^{\ell}+\sum_{j=1}^{\ell-1} b_j\ve^j\Bigr)=
(e_{i_1}- e_{i_0})\ve^{i_1}+\sum_{i=i_0+1}^{i_1-1} (e_i+e_{i_0}b_{i-i_0})\ve^i.$$

In both cases we see that the length of the sum for $e$ has decreased while all the coefficients remained non-negative. Continuing by induction, we prove the first part of the lemma.

For the second part, write $e=c\ve^i+d\ve^{i+1}$. By the four square theorem, we have $c=t_1^2+t_2^2+t_3^2+t_4^2$ and $d=t_5^2+t_6^2+t_7^2+t_8^2$  for some integers $t_j$. If $i$ is even, then taking $x_i=t_i\ve^{i/2}$ shows that $e$ is represented by the given form;
if $i$ is odd, one takes $x_i=t_{i+4}\ve^{(i+1)/2}$ for $i=1, \dots, 4$ and $x_i=t_{i-4}\ve^{(i-1)/2}$ for $i=5, \dots, 8$.
\end{proof}

\begin{theorem}\label{construct universal}
The quadratic form $$\sum_{\sigma\in S_0} \sigma(x_{1, \sigma}^2+x_{2, \sigma}^2+x_{3, \sigma}^2+x_{4, \sigma}^2+\ve x_{5, \sigma}^2+\ve x_{6, \sigma}^2+\ve x_{7, \sigma}^2+\ve x_{8, \sigma}^2)$$ is universal and has $8M_D$ variables. 
\end{theorem}

\begin{proof}
The universality of the form follows by Proposition \ref{express using semi-convergents} and Lemma \ref{kim}. Its number of variables is eight times the number of elements of $S_0$.

When $s$ is even, the elements of $S_0$ are of the form $\alpha_i+r\alpha_{i+1}$ with odd $i$, $-1\leq i\leq s-3$ and $0\leq r<u_{i+2}$. The same is true when $s$ is odd, except that we take $-1\leq i\leq 2s-3$, so that $|S_0| = M_D$. 
\end{proof}

Finally, let us note that this result is exactly analogous to Kim's case of $D=n^2-1$ \cite{Ki2}. Then we have $\sqrt D=[n-1, \overline{1, 2(n-1)}]$, and so $s=2$, $u_1=1$, and $S_0=\{1\}$, also obtaining $M_D=1$ and $m_{\text{diag}}(D) \leq 8$.

\section{Lower bound}

Let us now prove a lower bound on the number of variables of a diagonal universal quadratic form. Assume that $Q(x_i)=\sum_{1\leq i\leq m} a_{i}x_i^2$ is a universal totally positive diagonal quadratic form with $a_i\in\mathcal O_K^+$. For every indecomposable element $\sigma$ we know that
$\sigma=Q(x_i)$ for some $x_1, \dots, x_m$, but since $\sigma$ is indecomposable, this is possible only when $\sigma=a_ix_i^2$ for some $i$.
Hence it will be important for us to understand when it can happen that   squares of elements of $\co_K$ are indecomposable.

\begin{lemma}\label{convergent root}
Assume that $\alpha\in\mathcal O_K$ is such that $\alpha^2$ is indecomposable. Then for some $j \geq -1$ we have $ \alpha\in\{ \pm \alpha_j, \pm \alpha_j'\}$. 
\end{lemma}

\begin{proof} By the theorem of Dress-Scharlau \cite[Theorem 3]{DS} we know  
that $N(\alpha^2)\leq \Delta/4$, and so $|N(\alpha)|\leq \frac{1}{2}\sqrt \Delta$. Since $\alpha$ must obviously be primitive, the claim follows from Proposition  \ref{norm implies alpha}.
\end{proof}

Using the previous lemma, we are now ready to prove a lower bound on the number of variables.

\begin{theorem}\label{lower bound}
Every diagonal universal totally positive quadratic form over $\mathcal O_K$ needs at least $M_D/s$ variables if $s$ is even and $M_D/2s$ if $s$ is odd.
\end{theorem}

\begin{proof}
Let $T=S_0$ if $s$ is odd and $T=\{\sigma\in S \mid  \ve^2>\sigma\geq\sigma\pr>0\}$ if $s$ is even, i.e., we can describe $T$ uniformly as $T=\{\sigma\in S \mid \alpha_{s-1}^2=\alpha_{2s-1}>\sigma\geq\sigma\pr>0\}$. Thus $T$ has cardinality $2M_D$ if $s$ is even and $M_D$ if $s$ is odd. 

Note that if $1\neq\sigma=x-y\omega_D\pr\in T$, then $x, y> 0$ and $\sigma> 1> \sigma\pr$.

As before, let $Q(x_i)=\sum_{1\leq i\leq m} a_{i}x_i^2$ be a universal form with $a_i\in\mathcal O_K^+$. Multiplying any of the variables by a power of the unit $\alpha_{s-1}$ is an invertible substitution, and so we can assume without loss of generality that the coefficients $a_i$ satisfy $0<a_i<a_i\pr<\alpha_{s-1}^2$.

The only way in which $Q$ can represent an indecomposable $\sigma$ is when all but one of the variables $x_i$  are zero. 
Hence for each $\sigma\in T$ we have $\sigma=a_{i(\sigma)}\alpha(\sigma)^2$ for some index $1\leq i(\sigma)\leq m$ and some $\alpha(\sigma)\in\mathcal O_K$, which must be indecomposable. By  Lemma \ref{convergent root}, $\alpha(\sigma)$ is, up to sign, a convergent $\alpha_j$ or its conjugate $\alpha_j'$ for some $j$. Moreover, we have the two bounds 
 $$|\alpha(\sigma)|=(\sigma /a_{i(\sigma)})^{1/2}  \begin{cases} < \alpha_{s-1} \cdot (a'_{i(\sigma)})^{1/2} < \alpha_{s-1}^2,  \\  > (\sigma'/a'_{i(\sigma)})^{1/2}  = |\alpha(\sigma)\pr|. \end{cases} $$ In particular, this excludes the possibility $\alpha(\sigma) = \pm \alpha_j'$ and forces $\alpha(\sigma) =\pm  \alpha_j$ with  $-1\leq j\leq 2s-2$. 

Since all elements of $T$ are represented by $Q$, we have $T\subset\{a_i\alpha_j^2\mid  1\leq i\leq m, -1\leq j\leq 2s-2\}$. Comparing cardinalities, we get $\# T\leq m\cdot 2s$, finishing the proof.
\end{proof}

Comparing this bound on the number of variables with Theorem \ref{construct universal}, we see that it is asymptotically good when $s$ is small. This of course happens only rarely, and so it would be very interesting to be able to replace $1/s$ in Theorem \ref{lower bound} by an absolute constant.

This is perhaps hard, but we can almost achieve this, at the cost of having to consider the sum $M^*_{D, \varepsilon}$ of sufficiently large coefficients $u_j$ instead of $M_D$, as we are now going to prove.

\medskip

Let us first consider in more detail the set of semi-convergents $\alpha_{i, r}$ with fixed odd $i$ and varying $r$. Their number is $u_{i+2}\asymp \sqrt D/N_{i+1}$, and so if $N_{i+1}$ is small, there are many of them. 
Let $N:=N_{i+1}$, $T:=T_{i+1}$ and $u:= u_{i+2}$ so that
$$N(\alpha_{i,r}) = \begin{cases} \frac{D-(T-Nr)^2}N, & D \equiv 2, 3\, (\text{mod }4),\\  \frac{(D-1)/4+(T-Nr)-(T-Nr)^2}N, & D \equiv 1 \, (\text{mod 4})\end{cases}$$
by \eqref{norms}. 
In either case, we view this as a quadratic polynomial $f(r)$ and note that from the definition of $\alpha_{i,r}$ as (totally positive) indecomposables we have
$$\mathbb{Z} \cap [0, u] = \{n \in \mathbb{Z} \mid f(n) > 0\}$$
(using the estimate for $T$ from \eqref{lemma3}, this also corresponds to $\left|\frac TN-\frac u2\right|$ being small).
We define the multiplicative arithmetic function $\rho_f(d) = \{n \, (\text{{\rm mod }} d) \mid f(n) \equiv 0 \, (\text{{\rm mod }} d)\}$. An application of Hensel's lemma shows 
\begin{equation}\label{hensel}
  \rho_f(p^k) \leq 2
 \end{equation} 
   for all primes $p$ and all $k \geq 2$:

We verify this first in the case $f(x) = \frac{D-(T-Nx)^2}N=\frac{D-T^2}N+2Tx-Nx^2$. 

If $p\mid D$, then $D\equiv (T-Nx)^2\pmod {p^2}$ has no solution since $D$ is squarefree, i.e., $\rho(p^k)=0$ for all $k\geq 2$.

If $p\nmid D$ and $p\nmid N$, then $N$ is invertible modulo $p$, and so $\rho(p) \leq 2$ and if $p\neq 2$, we can use Hensel's lemma to conclude $\rho(p^k) \leq 2$ for every $k \geq 1$. If $p=2$, then we see directly that $\rho(2^k)=0$ for all $k\geq 2$, since $D\equiv 2,3\not\equiv 1\equiv (T-Nx)^2 \pmod 4$.

If $p\nmid 4D$ and $p\mid N$, then $f(x)\pmod p$ is linear or constant, and so $\rho(p)\leq 1$ and $\rho(p^k) \leq 1$ for every $k \geq 1$ by Hensel's lemma.

Finally, if $p=2\nmid D$ and $2\mid N$, then $f(x)\equiv \frac{D-T^2}N\pmod 2$. Since $D\equiv 2, 3\pmod 4$, we have  $4\nmid D-T^2$ and we see that $\rho(2^k)=0$ for all $k\geq 1$ in this case.

\medskip

In the second case we have $f(x)=\frac{D-\left(2(T-Nx)-1\right)^2}{4N}=\frac{D-(2T-1)^2}{4N}+(2T-1)x-Nx^2$. The discussion is similar to the previous one:

If $p\mid D$, then again $D\equiv\left(2(T-Nx)-1\right)^2\pmod {p^2}$ has no solution, i.e., $\rho(p^k)=0$ for all $k\geq 2$.

If $p\nmid D$ and $p\nmid N$, then $N$ is invertible modulo $p$, and so $\rho(p) \leq 2$ and if $p\neq 2$, we can use Hensel's lemma to conclude $\rho(p^k) \leq 2$ for every $k \geq 1$. For $p=2$ we have that $f(x)$ is constant modulo 2 and non-constant modulo 4, and so
$\rho(2^k)\leq 2$ for $k\geq 2$.

If $p\nmid 4D$ and $p\mid N$, then $f(x)\pmod p$ is linear or constant, and so $\rho(p)\leq 1$ and $\rho(p^k) \leq 1$ for every $k \geq 1$ by Hensel's lemma.

If $p=2\nmid D$ and $2\mid N$, then $f(x)\equiv \frac{D-(2T-1)^2}{4N}-x\pmod 2$, and so $\rho(2)=1$. Distinguishing some more cases, one gets $\rho(4)=1$ and $\rho(2^k)\leq 2$ for $k\geq 3$.

\medskip

Having established \eqref{hensel}, we can now prove  the following lemma. 
\begin{lemma} Let $f, T, N, u$ be as above. Let $k \in 2\mathbb{N}$ and $1 \leq X \leq u$. Then
$$\mathcal{S} := \# \{n \leq X \mid f(n) \text{ is $k$-th power free}\}  = C_{k, f} X + O\left( YX^{\varepsilon} + \frac{X}{Y^{k-1}} +   \frac{D^{1+\varepsilon}}{NY^k}\right)$$
for any $\varepsilon > 0$, where $C_{k, f} \geq 1/\zeta(k)^3$, and $1 \leq Y \leq X$ can be chosen arbitrarily. 
\end{lemma}

\begin{proof}  For $1 \leq Y \leq X$ we have
\begin{equation}\label{start}
\mathcal{S} =   \sum_{n \leq X} \sum_{d^k \mid f(n)} \mu(d) = \sum_{d \leq Y} \mu(d) \sum_{\substack{n \leq X\\ d^k \mid f(n)}} 1 + \sum_{\substack{n \leq X\\ f(n) = d^k m\\ d > Y}} \mu(d).
\end{equation}
The first term on the right hand side equals
$$ \sum_{d \leq Y} \mu(d) \left(\frac{\rho_f(d^k)}{d^k} X + O(X^{\varepsilon})\right) = C_{k, f} X + O(YX^{\varepsilon} + Y^{-k+1}X)$$
with
$$C_{k, f} = \sum_{d =1}^{\infty} \mu(d)  \frac{\rho_f(d^k)}{d^k} = \prod_{p} \left(1 - \frac{\rho_f(p^k)}{p^k}\right) \geq \prod_{p} \left( 1 - \frac{2}{p^k}\right) \geq \frac{1}{\zeta(k)^3}.$$
For the second term, we observe that  $0  < f(n) \leq  D/N$ for $1 \leq n \leq X$, so that     $0 < m \leq D/(NY^k)$. The summation condition $f(n) = d^km$ is equivalent to the   equation $(d^{k/2})^2mN + (T - Nn)^2 = D$ resp.\ $4(d^{k/2})^2mN + (2(T-Nn) - 1)^2 = D$. For given $m, N, T, D$, this has at most $6\tau(D)$ solutions $(d, n)$, since the imaginary quadratic number field $\mathbb{Q}(\sqrt{-mN})$ can have at most 6 units and the number of principal ideals of norm $D$ is at most $\tau(D) \ll D^{\varepsilon}$.   Hence we can bound the second term in \eqref{start} by $O(D^{1+\varepsilon} Y^{-k}N^{-1})$. \end{proof}

Now we are ready to bound the minimal number of variables from below in terms of the sum $M^*_{D, \varepsilon}$ of coefficients $u_i\geq D^{1/8+\varepsilon}$,
defined immediately before Theorem \ref{main 1}.

\begin{theorem}\label{lower C}
We have $$m_{\text{{\rm diag}}}(D) \gg M^*_{D, \varepsilon},$$ 
where the implicit constant depends only on $\varepsilon>0$.
\end{theorem}

\begin{proof}
Let's consider the (pairwise distinct) indecomposables $\alpha_{i, r}$ with odd $i$ and $0\leq r< u$.
Arguing similarly as in the proof of Theorem \ref{lower bound}, each of them has to be of the form $\sigma=a_{j(\sigma)}\alpha(\sigma)^2$, where $a_j$ are the coefficients of a universal form. 
Hence it suffices to show that for a positive density of suitable values of $i, r$  we have that the elements $\alpha_{i, r}$ have 4th-power-free norms (because then the indices $j(\alpha_{i,r})$ must be different for different values of $i, r$). Note that for $-1\leq i\leq 2s-3$, 
the indecomposables $\alpha_{i, r}$ are pairwise distinct modulo squares of units, and so we can consider the cases of different $i$'s separately and then just add them together.

Hence let $i$ be odd. As before we write $u:=u_{i+2}$, $N:=N_{i+1}$, and $T:=T_{i+1}$. 
We   apply the previous lemma  with $X = u$, $Y = u^{1-2\varepsilon}$ and $k=4$.  Provided $N \leq D^{3/8 - \varepsilon}$ (for $\varepsilon > 0$ sufficiently small and $D$ sufficiently large) we can conclude from \eqref{4b}   
that  $u \asymp D^{1/2}N^{-1}$ and $$\mathcal{S} \gg u + O(u^{1-\varepsilon} + D^{1+\varepsilon} N^{-1} u^{-4+8\varepsilon}) = u + O\left(u^{1-\varepsilon}(1 + N^{4-9\varepsilon} D^{-\frac{3}{2} + \frac{11}{2}\varepsilon})\right)    
\gg u$$
 as desired. 
\end{proof}

\section{Analytic preparation}

We now employ analytic methods to prove Theorems \ref{main 1a} and \ref{main 2}. 
We denote by $\mathcal{C}$ the class group of $K$ (of order $h$) and by 
 $\widehat{\mathcal{C}}$ the corresponding  character group. 
 For each $\chi \in \widehat{\mathcal{C}}$ we denote by $L_K(s, \chi)$ the corresponding class group $L$-function. It is entire except if $\chi = \chi_0$ is the trivial character, in which case $L_K(s, \chi_0) = \zeta(s) L(s, \chi_{\Delta})$  has a simple pole at $s=1$ with residue $L(1, \chi_{\Delta})$, We have the important subconvexity bound 
\begin{equation}\label{subconvex}
L_K(1/2 + it, \chi) \ll (1+ |t|)^A D^{\frac{1}{4} - \delta}
\end{equation}
where $A, \delta > 0$ are some absolute constants. The first such result was obtained in \cite{DFI}, the best value of $\delta = 1/1889$ is due to \cite{BHM}.  
We write
\begin{displaymath}
\begin{split}
L^{\ast}_K(s, \chi) & := \sum_{\mathfrak{a} \text{ primitive}} \frac{\chi(\mathfrak{a})}{(N\mathfrak{a})^s} = \prod_{p = \mathfrak{p} \bar{\mathfrak{p}} \text{ split}} \left(1 + \sum_{n=1}^{\infty} \frac{2\Re \chi(\mathfrak{p})}{p^{ns}}\right) \prod_{p = \mathfrak{p}^2 \text{ ramified} }\left(1 - \frac{\chi(\mathfrak{p})}{p^s}\right)^{-1}\\
& = L_K(s, \chi) \zeta^{(\Delta)}(2s)^{-1}
\end{split}
\end{displaymath}
where the superscript $(\Delta)$ denotes the removal of Euler factors at primes dividing $\Delta$. 

 For $s > 1$ we have
$$\frac{1}{h}\sum_{\chi \not= \chi_0} L_K(s, \chi) =\zeta(s, \text{princ})- \frac{1}{h} \zeta(s) L(s, \chi_{\Delta}),$$
where
$$\zeta(s,   \text{princ}) = \sum_{\mathfrak{a} \text{ principal}} \frac{1}{(N\mathfrak{a})^s}$$
is the zeta-function associated to the class of principal ideal. If we   write
\begin{equation}\label{LD}
L(D) = \lim_{s \rightarrow 1} \left[\zeta(s,   \text{princ}) - \frac{L(1, \chi_{\Delta})/h}{s-1}\right],
\end{equation}
then comparing Taylor coefficients in the equation
$$\frac{1}{h} \sum_{\chi \not= \chi_0} L_K(s, \chi) = \zeta(s, \text{princ}) - \frac{1}{h} L(s, \chi_{\Delta}) \zeta(s)$$
we obtain
\begin{equation*} 
\frac{1}{h}\sum_{\chi \not= \chi_0} L_K(1, \chi) = L(D) - \frac{\gamma L(1, \chi_{\Delta}) + L'(1, \chi_{\Delta})}{h},
\end{equation*}
where $\gamma = 0.577\ldots$ denotes Euler's constant, and so 
\begin{equation}\label{res}
\frac{1}{h}\sum_{ \chi \not= \chi_0} L^{\ast} _K(1, \chi) = \frac{1}{\zeta^{(\Delta)}(2)} \left(L(D) - \frac{\gamma L(1, \chi_{\Delta}) + L'(1, \chi_{\Delta})}{h}\right). 
\end{equation}
Explicit, but very complicated expressions for $L(D)$ have  been obtained by Hecke \cite{Heck}, Herglotz \cite{Her}, and Zagier \cite{Za} and can be regarded as a real-quadratic analogue of Kronecker's limit formula.

In the following we fix a smooth function $w: [0, \infty) \rightarrow [0, 1]$  that is 1 on $[0, 1]$ and 0 on $[2, \infty)$. Its Mellin transform $\widehat{w}(s) = \int_0^{\infty} w(x) x^{s-1}dx$ (initially in $\Re s > 0$) is entire  except for a simple pole at $s=0$ with residue 1, which can be seen by decomposing $w = w|_{[0, 1]} + w|_{(1, 2]}$. More precisely, it has a Laurent expansion
$$\widehat{w}(s) = \frac{1}{s} + c + \ldots, \quad c = \int_1^2 w(x) \frac{dx}{x}.$$
Moreover, it is rapidly decaying on vertical lines:
\begin{equation}\label{decay}
\widehat{w}(s) \ll_B |s|^{-B}
\end{equation}
for any constant $B \geq 1$ and $\Re s > -1$ (say), as can be seen by partial integration. 

\begin{lemma}\label{lemma15} Let $X \geq 1$.  Then
\begin{displaymath}
\begin{split}
\left.\sum_{\mathfrak{a}}\right.^{\ast}  \frac{1}{N\mathfrak{a}} w \left(\frac{N\mathfrak{a}}{X}\right)= \frac{1}{\zeta^{(\Delta)}(2)}\left( L(D) + \frac{L(1, \chi_{\Delta})(\log X + c + (\log \zeta^{(\Delta)})'(2))}{h}\right)+\\
+ O\left(X^{-\frac{1}{2}} D^{\frac{1}{4} - \frac{1}{2000}}\right),
\end{split}
\end{displaymath}
and
\begin{displaymath}
\begin{split}
& \left.\sum_{\mathfrak{a}}\right.^{\ast}   w \left(\frac{N\mathfrak{a}}{X}\right)= \frac{\widehat{w}(1)}{\zeta^{(\Delta)}(2)} \frac{X}{h} L(1, \chi_{\Delta}) + O(X^{\frac{1}{2}} D^{\frac{1}{4} - \frac{1}{2000}}),
\end{split}
\end{displaymath}
where as before $\sum^{\ast}$ denotes a sum over primitive principal ideals and   $(\log \zeta^{(\Delta)})'(2):=\frac d{ds} (\log \zeta^{(\Delta)}(s))\big\rvert_{s=2}$. 
\end{lemma}

\begin{proof} This is a standard contour shift argument. We have by Mellin inversion and orthogonality of class group characters
\begin{displaymath}
\begin{split}
\left.\sum_{\mathfrak{a}}\right.^{\ast} \frac{1}{N\mathfrak{a}} w \left(\frac{N\mathfrak{a}}{X}\right) = \int_{(2)} \frac{1}{h} \sum_{\chi \in \widehat{\mathcal{C}}} L^{\ast}_K(s+1, \chi)\widehat{w}(s) X^s \frac{ds}{2\pi i}
\end{split}
\end{displaymath}
where here and in the following $\int_{(c)}$ denotes a complex contour integral over the vertical line $\Re s = c$. 
Shifting the contour to $\Re s = -1/2$, computing the residue of the double pole at $s=0$,  namely 
\begin{equation}\label{pole}
\begin{split}
&\frac{1}{h} \sum_{\chi \not= \chi_0} L^{\ast}_K(1, \chi) + \frac{L(1, \chi_{\Delta})(\log X + c +  (\log \zeta^{(\Delta)})'(2) + \gamma ) + L'(1, \chi_{\Delta})}{\zeta^{(\Delta)}(2) h} = \\
&=\frac{1}{\zeta^{(\Delta)}(2)}\left( L(D) + \frac{L(1, \chi_{\Delta})(\log X + c + (\log \zeta^{(\Delta)})'(2))}{h}\right)
\end{split}
\end{equation}
by  \eqref{res},  
and estimating the remaining integral with \eqref{subconvex} and \eqref{decay} gives the first formula. The proof of  the second formula is similar.  \end{proof}

The following lemma  is an application of Burgess' bound. 

\begin{lemma}\label{burgess} For $\varepsilon > 0$ we have
$$L'(1, \chi_{\Delta}) \geq - \left(\frac{3}{8} + \varepsilon\right) L(1, \chi_{\Delta})\left( \log D  + O_{\varepsilon}(1)\right).$$
\end{lemma}

\begin{proof} Let 
$$r(n) := \sum_{d \mid n} \chi_{\Delta}(d).$$
For $X \geq 1$ we compute by Mellin inversion
$$\mathcal{S} := \sum_{n \leq X} \frac{r(n)}{n} \left(1 - \frac{n}{X}\right) = \int_{(2)} \zeta(s+1) L(s+1, \chi_{\Delta}) \frac{X^s}{s(s+1)} \frac{ds}{2\pi i}.$$
We shift the contour to $\Re s = -1/2$. The residue at $s=0$ contributes
\begin{equation}\label{one}
L(1, \chi_{\Delta})( \log X + \gamma - 1) + L'(1, \chi_{\Delta}),
\end{equation}
while Heath-Brown's hybrid bound \cite{HB}
\begin{equation}\label{hybrid}
L(1/2 + it, \chi_{\Delta}) \ll ((1 + |t|)D)^{3/16+\varepsilon}
\end{equation}
along with the classical convexity estimate  $\zeta(1/2 + it ) \ll (1 + |t|)^{1/4+\varepsilon}$ bounds the remaining integral by
\begin{equation}\label{two}
O(X^{-1/2} D^{3/16+\varepsilon}).
\end{equation}
On the other hand, for a parameter $1 \leq Y \leq X/10$ we have by partial summation
$$\mathcal{S} \geq \sum_{Y \leq n \leq X} \frac{r(n)}{n} \left(1 - \frac{n}{X}\right) = \int_Y^X \sum_{Y \leq n \leq t} r(n) \frac{dt}{t^2} \geq \int_{3Y}^X \sum_{Y \leq n \leq t} r(n) \frac{dt}{t^2}. $$
For $0 < \varepsilon < 1/10$ and $t \geq 3Y$ let $V_{t} : \mathbb{R} \rightarrow [0, 1]$ be a smooth function with  support in $[Y, t]$ that is 1 on $[2Y, (1-\varepsilon)t]$ and satisfies $V_t^{(j)}(x) \ll_{\varepsilon, j} x^{-j}$ for $j \in \mathbb{N}_0$ uniformly in $t$ and $Y$. It follows from the definition and repeated integration by parts that
\begin{equation}\label{boundV}
\widehat{V}_t(s) = \int_0^{\infty} V_t(s) x^{s-1} dx \ll t^{\Re s} |s|^{-10}
\end{equation}
for $\Re s  \geq  1/2$. 
Then again by Mellin inversion we have
\begin{displaymath}
\begin{split}
\mathcal{S}& \geq \int_{3Y}^X  \sum_n  r(n)V_t(n) \frac{dt}{t^2} =  \int_{3Y}^X \int_{(2)} \zeta(s) L(s, \chi_{\Delta}) \widehat{V}_t(s) \frac{ds}{2\pi i} \frac{dt}{t^2}.
\end{split}
\end{displaymath}
Shifting the $s$-contour to $\Re s = 1/2$ and using \eqref{hybrid} and \eqref{boundV}, we obtain
\begin{displaymath}
\begin{split}
\mathcal{S}& \geq  \int_{3Y}^X \left( \Bigr(t(1-\varepsilon) - 2Y\Bigl) L(1, \chi_{\Delta}) + O(t^{1/2} D^{3/16+\varepsilon})\right) \frac{dt}{t^2}\\
& = L(1, \chi_{\Delta}) \left( (1 - \varepsilon) \log \frac{X }{3Y} + O(1)\right) + O(D^{3/16+\varepsilon} Y^{-1/2}). 
\end{split}
\end{displaymath}
Combining this with \eqref{one} and \eqref{two} yields
$$L(1, \chi_{\Delta})(\log X + O(1)) + L'(1, \chi_{\Delta}) \geq L(1, \chi_{\Delta}) \left(  (1 - \varepsilon) \log \frac{X}{Y} + O(1)\right)  + O(D^{3/16+\varepsilon} Y^{-1/2}).$$
Now we choose
$$Y = D^{3/8 + 4\varepsilon}, \quad X = D\,\, (\text{say})$$
getting
$$L'(1, \chi_{\Delta}) \geq - \left(\frac{3}{8} + 5\varepsilon\right)L(1, \chi_{\Delta}) \left(\log D  + O(1)\right) + O(D^{-\varepsilon}).$$
By Siegel's theorem, the last error term is negligible compared to the main term, and the result follows. \end{proof}

\section{Bounds for the sum of convergents}\label{sec:main estimate}

\begin{theorem}\label{asymp} We have
$$\sum_{i=1}^s u_i = \frac{ \sqrt{\Delta} }{\zeta^{(\Delta)}(2)} \Bigl(L(D)  +  \frac{1}{h} L(1, \chi_{\Delta})(\log \sqrt{D} + O(1)) \Bigr)    + O(D^{\frac{1}{2} - \frac{1}{2000}} ).$$
\end{theorem}

\begin{proof} Using Corollary \ref{estimate sum} we prove an upper bound and a lower bound. On the one hand  we have 
\begin{displaymath}
\begin{split}
\sum_{i=1}^s u_i & \leq \underset{N\mathfrak{a} \leq   \sqrt{\Delta}}{\left.\sum \right.^{\ast}}  \frac{  \sqrt{\Delta}}{N \mathfrak{a}}  \leq \left.\sum_{\mathfrak{a}}\right.^{\ast}   \frac{  \sqrt{\Delta}}{N \mathfrak{a}}w\left( \frac{N\mathfrak{a}}{ \sqrt{\Delta}}\right),
\end{split}
\end{displaymath}
so that
$$\sum_{i=1}^s u_i \leq   \frac{ \sqrt{\Delta} }{\zeta^{(\Delta)}(2)}  \Bigl(L(D)  +  \frac{1}{h} L(1, \chi_{\Delta})(\log \sqrt{D} + O(1)) \Bigr)     + O(D^{\frac{1}{2} - \frac{1}{2000}} )$$
by Lemma \ref{lemma15}. 
On the other hand, since $\text{supp}(w) \subseteq [0, 2]$,  we have 
 \begin{displaymath}
\begin{split}
\sum_{i=1}^s u_i & \geq 
\left.\sum_{\mathfrak{a}} \right.^{\ast} \frac{ \sqrt{\Delta}}{N\mathfrak{a}} w \left(\frac{2 N\mathfrak{a}}{\frac{1}{2}  \sqrt{\Delta}}\right) + O\left(\left.\sum_{\mathfrak{a}} \right.^{\ast}  w\left(\frac{N\mathfrak{a}}{ \sqrt{\Delta}}\right) \right),
 \end{split}
\end{displaymath}
getting a corresponding lower bound again by Lemma \ref{lemma15}. 
\end{proof}

Since $u_s = 2\lfloor\omega_D\rfloor$ or $2\lfloor\omega_D\rfloor-1$, it is clear that the error term $O(D^{\frac{1}{2} - \frac{1}{2000}} )$ is smaller than the main term. However, since $L(D)$ may be negative, it is not clear if $L(D) + \log \sqrt{D} L(1, \chi_{\Delta})/h$ grows more quickly than   $L(1, \chi_{-D})/h$ as $D \rightarrow \infty$, so that a priori Theorem \ref{asymp} may only give an upper bound. In order to understand the behaviour of the main term, it is more convenient to use the left hand side of the expression \eqref{pole} with $X = \sqrt{D}$. Notice that all $L$-values are real (since $L(s, \chi) = L(s, \bar{\chi}) = \overline{L(s, \chi)}$ for $\chi \in \widehat{\mathcal{C}}$ and $s\in \mathbb{R}$) and, except for $L'(1, \chi_{\Delta})$, positive and that $(\log \zeta^{(\Delta)})'(2) \asymp 1$.  Since  Lemma \ref{burgess} implies 
$$ L(1, \chi_{\Delta}) \log \sqrt{D} + L'(1, \chi_{\Delta}) \gg L(1, \chi_{\Delta}) \log \sqrt{D},$$
for sufficiently large $D$, \eqref{pole} implies 
\begin{equation}\label{lower}
L(D)  +  \frac{1}{h} L(1, \chi_{\Delta})(\log \sqrt{D} + O(1)) \gg \frac{1}{h} L(1, \chi_{\Delta})\log \sqrt{D},
\end{equation}
for sufficiently large $D$, and hence
\begin{cor}\label{Korollar} We have
$$\sum_{i=1}^s u_i \sim  \frac{\sqrt{\Delta} }{\zeta^{(\Delta)}(2)}   \Bigl(L(D)  +  \frac{1}{h} L(1, \chi_{\Delta})\log \sqrt{D} \Bigr) $$
as $D \rightarrow \infty$. 
\end{cor}

We have the classical upper bounds \cite[Lemma 4]{Fo}
\begin{equation}\label{fogels}
L(s, \chi_{\Delta}) \ll \log \Delta, \quad L_K(s, \chi) \ll (\log \Delta)^2
\end{equation}
for $|s-1| \leq 1/\log \Delta$ and $\chi \in \widehat{\mathcal{C}}$, from which we also conclude $L'(1, \chi_{\Delta}) \ll (\log \Delta)^2$ by Cauchy's integral formula. 
Thus we obtain
$$\sqrt{\Delta} \leq  2\lfloor\omega_D\rfloor \leq   \sum_{i=1}^s u_i \ll  \sqrt{\Delta} (\log \Delta)^2 $$
from \eqref{pole}, which by \eqref{MD} and \eqref{s odd} implies the upper bound in  Theorem \ref{main 1a}. As mentioned in the introduction, the lower bound $M^{\ast}_{D, \varepsilon} \geq \sqrt{\Delta}$ is trivial if $s$ is odd. 

\

\textbf{Concluding remarks:} 

(1) If one is willing to assume the Generalized Riemann Hypothesis, then $L(1, \chi_{\Delta}) \gg 1/\log\log D$ \cite{Li},  and instead of \eqref{fogels} we have
$$L(1, \chi_{\Delta}), L_K(1, \chi) \ll \log\log D, \quad L'(1, \chi_{\Delta}) \ll (\log\log D)^2$$
(cf.\ \cite[Section 5]{Li1}), so that again by \eqref{pole} for the upper bound and \eqref{lower} for the lower bound we obtain
 \begin{equation}\label{GRH}
 \sqrt{D} \left( 1+ \frac{\log D}{h \log\log D}\right)  \ll  \sum_{i=1}^s u_i \ll  \sqrt{D} \left( 1+ \frac{\log D}{h}\right) (\log\log D)^2.
 \end{equation}
 
 \
 
(2) Using the second part of Lemma \ref{lemma15}, we obtain easily
\begin{equation}\label{estimate s}
\frac{\sqrt{D}}{h} L(1, \chi_{\Delta}) + O(D^{\frac{1}{2} - \frac{1}{2000}})
 \ll s \ll \frac{\sqrt{D}}{h} L(1, \chi_{\Delta}) + O(D^{\frac{1}{2}  -\frac{1}{2000}}),
 \end{equation}
 where the lower bound is of course a trivial statement if $h \geq D^{1/2000}$ (but conjecturally this does not happen very frequently).
 
 \

(3) In the same way as in Theorem \ref{asymp} one can show
$$M_D = \frac{  \sqrt{\Delta} }{\zeta^{(\Delta)}(2)} \Bigl(L^-(D)  +  \frac{1}{h^+} L(1, \chi_{\Delta})(\log \sqrt{D} + O(1)) \Bigr)    + O(D^{\frac{1}{2} - \frac{1}{2000}} )$$
where $h^+$ is the narrow class number (which equals $h$ if $\mathcal{O}_K$ has a unit of negative norm and otherwise equals $2 h$) and $L^-(D)$ is the constant Taylor coefficient of 
$\zeta(s,   \text{princ}^{-}) =   \sum_{\mathfrak{a}}  (N\mathfrak{a})^{-s}.$
 Alternatively, one can also combine Corollary \ref{Korollar} with \cite[Satz 2, \S 14]{Za2} for the alternating sum, to obtain a formula form for $M_D$. In many cases, however, the main term will be smaller than the error term, so   that it is unclear to which extent this formula is meaningful.

\end{document}